\documentclass[11pt,a4paper]{article}%
\usepackage{makeidx}
\usepackage{amsfonts}
\usepackage{amsmath}
\usepackage{amssymb}
\usepackage{graphicx}
\usepackage{palatino}
\usepackage{extarrows}
\usepackage[colorlinks]{hyperref}
\usepackage{mathtools}%
\providecommand{\U}[1]{\protect\rule{.1in}{.1in}}
\hypersetup{colorlinks=true, linkcolor=cyan, citecolor=green, filecolor=black, urlcolor=blue}
\newtheorem{theorem}{Theorem}

\newtheorem{corollary}[theorem]{Corollary}

\newtheorem{definition}[theorem]{Definition}

\newenvironment{proof}[1][Proof]{\noindent\textbf{#1.} }{\ \rule{0.5em}{0.5em}}
\begin{document}

\title{On the maximal monotonicity of subdifferential operators}
\author{Aurel R\u{a}\c{s}canu}
\date{}
\maketitle

\begin{abstract}
We present a simple proof of the maximal monotonicity of the subdifferential
operator in general Banach spaces. Using the Fitzpatrick function the
Rockafellar surjectivity theorem follows as a corollary.

\end{abstract}

\section{Introduction}

In this note, we give a very simple proof of Rockafellar's maximal
monotonicity theorem based on Ekeland's variational principle. The paper is
the result of particular discussions and remarks of Prof. C. Zalinescu on the
simplified proofs of Rockafellar's results: maximal monotonicity theorem and
surjectivity theorem. The proof for maximal monotonicity that we present here
comes from a note on the maximal monotonicity of the subdifferential operator
of the convex l.s.c. function $\Phi:C([0,T];\mathbb{R}^{d})\rightarrow
]-\infty,+\infty]$%
\[
\Phi(x)=\left\{
\begin{array}
[c]{l}%
{\displaystyle\int_{0}^{T}}
\varphi(x(t)dt,\,\,\text{\text{if }}\varphi\left(  x\right)  \in
L^{1}(0,T)\vspace{0.02in}\\
+\infty,\;\quad\quad\quad\text{\text{otherwise}}%
\end{array}
\right.
\]
given by Asiminoaei and R\u{a}\c{s}canu in \cite{AsRa:97}.

Remark that the first proof of maximal monotonicity theorem was given by
Rockafellar in \cite{Rock}. Other different and simplified proofs are given by
S. Simons in \cite{Si} and M. Marques Alves and B. F. Svaiter in \cite{MaSv}.

Let $\left(  \mathbb{X},\left\Vert .\right\Vert \right)  $ be a real Banach
space and $\mathbb{X}^{\ast}$ be its dual. For $x^{\ast}\in\mathbb{X}^{\ast}$
and $x\in\mathbb{X}$ we denote $x^{\ast}\left(  x\right)  $ (the value of
$x^{\ast}$ in $x$) by $\left\langle x,x^{\ast}\right\rangle $ or $\left\langle
x^{\ast},x\right\rangle .$ The space $\mathbb{X\times X}^{\ast}$ is also a
Banach space with the norm $\left\Vert \left(  x,x^{\ast}\right)  \right\Vert
_{\mathbb{X\times X}^{\ast}}=\left(  \left\Vert x\right\Vert ^{2}+\left\Vert
x^{\ast}\right\Vert _{\ast}^{2}\right)  ^{1/2}.$

If $A:\mathbb{X}\rightrightarrows\mathbb{X}^{\ast}$ is a point-to-set operator
(from $\mathbb{X}$ to the family of subsets of $\mathbb{X}^{\ast}$), then
$Dom\left(  A\right)  \xlongequal{\hspace{-4pt}{\rm def}\hspace{-4pt}}\left\{
x\in\mathbb{X}:A\left(  x\right)  \neq\emptyset\right\}  $ and $R\left(
A\right)  \xlongequal{\hspace{-4pt}{\rm def}\hspace{-4pt}}\{x^{\ast}%
\in\mathbb{X}^{\ast}:\exists~x\in Dom\left(  A\right)  $ s.t. $x^{\ast}\in
A\left(  x\right)  \}$. We say that $A$ is proper if $Dom\left(  A\right)
\neq\emptyset.$

We shall use the notation $\left(  x,x^{\ast}\right)  \in A$ for $x\in
Dom\left(  A\right)  $ and $x^{\ast}\in A\left(  x\right)  ;$ this means that
the operator $A$ is identified with its graph%
\[
gr\left(  A\right)  =\{\left(  x,x^{\ast}\right)  \in\mathbb{X\times X}^{\ast
}:x\in Dom\left(  A\right)  ,\;x^{\ast}\in A\left(  x\right)  \}.
\]

The operator $A:\mathbb{X}\rightrightarrows\mathbb{X}^{\ast}$ is monotone
($A\subset\mathbb{X\times X}^{\ast}$ is a monotone set) if%
\[
\left\langle x-y,x^{\ast}-y^{\ast}\right\rangle \geq0,\ \forall\left(
x,x^{\ast}\right)  ,\left(  y,y^{\ast}\right)  \in A.
\]
A monotone operator (set) is maximal monotone if its graph is not properly
contained in the graph of any other monotone operator (set). Hence a monotone
operator is maximal monotone if and only if%
\[
\inf\left\{  \left\langle x-u,x^{\ast}-u^{\ast}\right\rangle :\left(
u,u^{\ast}\right)  \in A\right\}  \geq0\quad\Longrightarrow\quad\left(
x,x^{\ast}\right)  \in A.
\]

Given a function $\varphi:\mathbb{X\rightarrow]-}\infty,+\infty],$ we denote
$Dom\left(  \varphi\right)
\xlongequal{\hspace{-4pt}{\rm def}\hspace{-4pt}}\{x\in\mathbb{X}%
:\varphi\left(  x\right)  <\infty\}.$ We say that $\varphi$ is proper if
$Dom\left(  \varphi\right)  \neq\emptyset.$ The subdifferential $\partial
\varphi:\mathbb{X}\rightrightarrows\mathbb{X}^{\ast}$ is defined by%
\[
\left(  x,x^{\ast}\right)  \in\partial\varphi\quad\text{if}\quad\left\langle
y-x,x^{\ast}\right\rangle +\varphi\left(  x\right)  \leq\varphi\left(
y\right)  ,\ \forall~y\in\mathbb{X}.
\]
Clearly if $\varphi$ is proper and $\left(  x,x^{\ast}\right)  \in
\partial\varphi$ then $\varphi\left(  x\right)  \in\mathbb{R},$ that is $x\in
Dom\left(  \varphi\right)  .$

Remark that if $\varphi:\mathbb{X\rightarrow]-}\infty,+\infty]$ is a proper
convex l.s.c. function and $\psi:\mathbb{X\rightarrow R}$ is a continuous
convex function then $\partial\left(  \varphi+\psi\right)  \left(  x\right)
=\partial\varphi\left(  x\right)  +\partial\psi\left(  x\right)  $ for all
$x\in Dom\left(  \partial\varphi\right)  $ (this result is a consequence of
the Theorem 2.8.7 from \cite{z-book}).

\section{The results}

\begin{theorem}
Let $\mathbb{X}$ be a Banach space and $\varphi:\mathbb{X}\rightarrow
]-\infty,+\infty]$ be a proper convex lower semicontinuous function. Then
$\partial\varphi:\mathbb{X}\rightrightarrows\mathbb{X}^{\ast}$ is proper
maximal monotone operator.
\end{theorem}

\begin{proof}
Using the definition of $\partial\varphi$ it is very easy to prove that
$\partial\varphi:\mathbb{X}\rightrightarrows\mathbb{X}^{\ast}$\textit{ }is a
monotone\textit{ }operator. Let us prove that $\partial\varphi$\textit{ }is a
proper maximal monotone operator.

Let $(z,z^{\ast})\in\mathbb{X}\times\mathbb{X}^{\ast}$ and $\lambda>0$ be
arbitrary fixed. Consider the function $\Psi:\mathbb{X}\rightarrow
]-\infty,+\infty]$ be defined by
\[
\Psi(x)=\dfrac{1}{2}\left\Vert x-z\right\Vert ^{2}+\lambda\varphi
(x)-\lambda\left\langle x,z^{\ast}\right\rangle .
\]
Then $\Psi$ is a proper convex lower semicontinuous function and $\Psi$ is
bounded from below. From Ekeland principle \cite{EkOT} (see also \cite[Th.
1.4.1]{z-book}, or \cite{BaMM}, p. 29, Th. 3.2), for every $\varepsilon>0$
there exists
\
$x_{\varepsilon}\in\mathbb{X}$ such that%
\begin{align}
\Psi(x_{\varepsilon})  &  \leq\inf\,\{\Psi(x):x\in\mathbb{X}\}+\varepsilon
^{2}\quad\text{and}\label{1a}\\
\Psi(x_{\varepsilon})  &  \leq\Psi(x)+\varepsilon~\left\Vert x-\mathit{\ }%
x_{\varepsilon}\right\Vert _{\mathbb{X}}~,\;\text{for all}\,x\in\mathbb{X}.
\label{1b}%
\end{align}
Remark that the sequence $\left\{  x_{\varepsilon}:0<\varepsilon\leq1\right\}
$ is bounded, since $\lim\limits_{\left\Vert x\right\Vert \rightarrow\infty
}\Psi(x)=+\infty.$ From (\ref{1b}) we deduce that%
\[
0\in\partial\Psi(x_{\varepsilon})+\varepsilon U_{\mathbb{X}^{\ast}%
}=J_{\mathbb{X}}(x_{\varepsilon}-z)+\lambda\partial\varphi(x_{\varepsilon
})-\lambda z^{\ast}+\varepsilon U_{\mathbb{X}^{\ast}}~,
\]
where $U_{\mathbb{X}^{\ast}}=\left\{  u^{\ast}\in\mathbb{X}^{\ast}:\left\Vert
u^{\ast}\right\Vert _{\mathbb{X}^{\ast}}\leq1\right\}  $ and $J_{\mathbb{X}%
}:\mathbb{X}\rightrightarrows\mathbb{X}^{\ast}$ is the duality mapping, that
is%
\[
J_{\mathbb{X}}\left(  x\right)  =\left\{  x^{\ast}\in\mathbb{X}^{\ast
}:\left\langle x,x^{\ast}\right\rangle =\left\Vert x\right\Vert ^{2}%
=\left\Vert x^{\ast}\right\Vert _{\mathbb{X}^{\ast}}^{2}\right\}
=\partial\left(  \frac{1}{2}\left\Vert \cdot\right\Vert ^{2}\right)  \left(
x\right)  .
\]
So, there exist $u_{\varepsilon}^{\ast}\in U_{X^{\ast}}$, $y_{\varepsilon
}^{\ast}\in J_{\mathbb{X}}(x_{\varepsilon}-z)$ and $x_{\varepsilon}^{\ast}%
\in\partial\varphi(x_{\varepsilon})$ (in particular $\partial\varphi$ is a
proper point-to-set operator) such that
\begin{equation}
\lambda z^{\ast}-\lambda x_{\varepsilon}^{\ast}=y_{\varepsilon}^{\ast
}+\varepsilon~u_{\varepsilon}^{\ast}~. \label{e1}%
\end{equation}
It follows that $\left\Vert \lambda z^{\ast}-\lambda x_{\varepsilon}^{\ast
}\right\Vert \leq\left\Vert y_{\varepsilon}^{\ast}\right\Vert +\varepsilon
=\left\Vert x_{\varepsilon}-z\right\Vert +\varepsilon.$

Let now $\left(  z,z^{\ast}\right)  \in\mathbb{X}\times\mathbb{X}^{\ast}$ such
that $\left\langle z-x,z^{\ast}-x^{\ast}\right\rangle \geq0$ for all $\left(
x,x^{\ast}\right)  $ $\in\partial\varphi.$ Then%
\[%
\begin{array}
[c]{ll}%
0\leq\left\langle z-x_{\varepsilon},z^{\ast}-x_{\varepsilon}^{\ast
}\right\rangle  & =\left\langle z-x_{\varepsilon},y_{\varepsilon}^{\ast
}\right\rangle +\varepsilon\left\langle z-x_{\varepsilon},u_{\varepsilon
}^{\ast}\right\rangle \medskip\\
& \leq-\left\Vert x_{\varepsilon}-z\right\Vert ^{2}+\varepsilon\left\Vert
x_{\varepsilon}-z\right\Vert .
\end{array}
\]
Hence $\left\Vert x_{\varepsilon}-z\right\Vert \leq\varepsilon$ and
$\left\Vert x_{\varepsilon}^{\ast}-z^{\ast}\right\Vert \leq\dfrac
{2\varepsilon}{\lambda}$; in particular $x_{\varepsilon}\rightarrow z$ and
$x_{\varepsilon}^{\ast}\rightarrow z^{\ast}$ as $\varepsilon\rightarrow0$.
Passing to $\liminf_{\varepsilon\rightarrow0}$ in
\[
\left\langle x_{\varepsilon}^{\ast},y-x_{\varepsilon}\right\rangle
+\varphi(x_{\varepsilon})\leq\varphi(y),\quad\forall\,y\in\mathbb{X},
\]
we obtain $\left(  z,z^{\ast}\right)  \in\partial\varphi$. \bigskip\hfill
\end{proof}

\bigskip

From this proof (the equality (\ref{e1}) corresponding to $z=0$) we deduce a
Rockafellar's type surjectivity result in general Banach spaces:

\begin{corollary}
\label{c1}If $\mathbb{X}$ is a Banach space and $\varphi:\mathbb{X}%
\rightarrow]-\infty,+\infty]$ is a proper convex lower semicontinuous function
then for all $\lambda>0,$%
\[
\overline{R\left(  J_{\mathbb{X}}+\lambda\partial\varphi\right)  }%
=\mathbb{X}^{\ast}.
\]
If $\mathbb{X}$ is a reflexive Banach space, then%
\[
R\left(  J_{\mathbb{X}}+\lambda\partial\varphi\right)  =\mathbb{X}^{\ast}.
\]

\end{corollary}

\begin{proof}
From (\ref{e1}) with $z=0$ we deduce that $R\left(  \partial\varphi+\lambda
J_{\mathbb{X}}\right)  +\varepsilon U_{\mathbb{X}^{\ast}}=\mathbb{X}^{\ast}.$
Hence%
\[
\overline{R\left(  J_{\mathbb{X}}+\lambda\partial\varphi\right)  }%
=\bigcap_{\varepsilon>0}\left[  R\left(  J_{\mathbb{X}}+\lambda\partial
\varphi\right)  +\varepsilon U_{\mathbb{X}^{\ast}}\right]  =\mathbb{X}^{\ast
}.
\]
If $\mathbb{X}$ is a reflexive Banach space, then the boundedness of $\left\{
x_{\varepsilon}:0<\varepsilon\leq1\right\}  $ yields that there exists a
subsequence $\varepsilon_{n}\rightarrow0$ such that $x_{\varepsilon_{n}%
}\rightharpoonup x_{0}$ (weakly on $\mathbb{X}$) and%
\[
\Psi(x_{0})\leq\liminf_{\varepsilon_{n}\rightarrow0}\Psi(x_{\varepsilon_{n}%
})=\inf\,\{\Psi(x):x\in\mathbb{X}\}.
\]
Hence $0\in\partial\Psi(x_{0})=J_{\mathbb{X}}(x_{0})+\lambda\partial
\varphi(x_{0})-\lambda z^{\ast}$ that is $\lambda z^{\ast}\in J_{\mathbb{X}%
}(x_{0})+\lambda\partial\varphi(x_{0}).$ Hence $\mathbb{X}^{\ast}%
=\lambda\mathbb{X}^{\ast}\subset R\left(  J_{\mathbb{X}}+\lambda
\partial\varphi\right)  \subset\mathbb{X}^{\ast}.$ \bigskip\hfill
\end{proof}

\begin{definition}
Given a monotone operator $A:\mathbb{X}\rightrightarrows\mathbb{X}^{\ast}$,
the associated Fitzpatrick function is defined as $\mathcal{H}=\mathcal{H}%
_{A}:\mathbb{X}\times\mathbb{X}^{\ast}\rightarrow\mathbb{]-}\infty,+\infty]$,%
\begin{equation}%
\begin{array}
[c]{ll}%
\mathcal{H}\left(  x,x^{\ast}\right)   &
\xlongequal{\hspace{-4pt}{\rm def}\hspace{-4pt}}\left\langle x,x^{\ast
}\right\rangle -\inf\left\{  \left\langle x-u,x^{\ast}-u^{\ast}\right\rangle
:\left(  u,u^{\ast}\right)  \in A\right\}  \smallskip\\
& =\sup\left\{  \left\langle u,x^{\ast}\right\rangle +\left\langle x,u^{\ast
}\right\rangle -\left\langle u,u^{\ast}\right\rangle :\left(  u,u^{\ast
}\right)  \in A\right\}
\end{array}
\label{F-def}%
\end{equation}

\end{definition}

\noindent Clearly $\mathcal{H}\left(  x,x^{\ast}\right)  =\left\langle
x,x^{\ast}\right\rangle ,\ $for all $(x,x^{\ast})\in A$ and%
\[
\mathcal{H}=\mathcal{H}_{A}:\mathbb{X}\times\mathbb{X}^{\ast}\rightarrow
\mathbb{]-}\infty,+\infty]\quad\text{{\normalsize is a proper convex l.s.c.
function.}}%
\]
Let $\left(  x^{\ast},x\right)  \in\partial\mathcal{H}\left(  u,u^{\ast
}\right)  .$ Then, from the definition of a subdifferential operator this
means that%
\[
\left\langle \left(  x^{\ast},x\right)  ,\left(  z,z^{\ast}\right)  -\left(
u,u^{\ast}\right)  \right\rangle +\mathcal{H}\left(  u,u^{\ast}\right)
\leq\mathcal{H}\left(  z,z^{\ast}\right)  ,\ \forall~\left(  z,z^{\ast
}\right)  \in\mathbb{X}\times\mathbb{X}^{\ast},
\]
or, equivalently,%
\begin{equation}%
\begin{array}
[c]{l}%
\left\langle u-x,u^{\ast}-x^{\ast}\right\rangle -\inf\limits_{\left(
y,y^{\ast}\right)  \in A}\left\langle u-y,u^{\ast}-y^{\ast}\right\rangle
\medskip\\
\leq\left\langle z-x,z^{\ast}-x^{\ast}\right\rangle -\inf\limits_{\left(
y,y^{\ast}\right)  \in A}\left\langle z-y,z^{\ast}-y^{\ast}\right\rangle
,\quad\forall~\left(  z,z^{\ast}\right)  \in\mathbb{X}\times\mathbb{X}^{\ast}.
\end{array}
\label{F3}%
\end{equation}
Since the operator $A$ is maximal monotone, then%
\[
\inf\limits_{\left(  y,y^{\ast}\right)  \in A}\left\langle u-y,u^{\ast
}-y^{\ast}\right\rangle \leq0
\]
and%
\[
\inf\limits_{\left(  y,y^{\ast}\right)  \in A}\left\langle z-y,z^{\ast
}-y^{\ast}\right\rangle =0,\ \forall~\left(  z,z^{\ast}\right)  \in A;
\]
consequently, we have%
\begin{equation}
\left(  x^{\ast},x\right)  \in\partial\mathcal{H}\left(  u,u^{\ast}\right)
\ \Longrightarrow\ \left\langle u-x,u^{\ast}-x^{\ast}\right\rangle \leq
\inf\limits_{\left(  z,z^{\ast}\right)  \in A}\left\langle z-x,z^{\ast
}-x^{\ast}\right\rangle \text{.}\label{F2}%
\end{equation}
Also, by the monotonicity of $A$, from (\ref{F3}) follows%
\[
\left(  x^{\ast},x\right)  \in A\ \Longrightarrow\ \left(  x^{\ast},x\right)
\in\partial\mathcal{H}\left(  x,x^{\ast}\right)  .
\]

\noindent Hence, if $A:\mathbb{X}\rightrightarrows\mathbb{X}^{\ast}$ is a
maximal monotone operator then $\mathcal{H}_{A}$ characterizes $A$ as follows.

\begin{theorem}
[Fitzpatrick](see Fitzpatrick \cite{F88} and Simons-Z\u{a}linescu \cite{SZ})
Let $\mathbb{X}$ be a Banach space, $A:\mathbb{X}\rightrightarrows
\mathbb{X}^{\ast}$ be a maximal monotone operator and $\mathcal{H}$ its
associated Fitzpatrick function. Then, for all $(x,x^{\ast})\in\mathbb{X}%
\times\mathbb{X}^{\ast}$%
\[
\mathcal{H}(x,x^{\ast})\geq\left\langle x,x^{\ast}\right\rangle .
\]
Moreover, the following assertions are equivalent%
\[%
\begin{array}
[c]{ll}%
\left(  a\right)  & \displaystyle(x,x^{\ast})\in A;\medskip\\
\left(  b\right)  & \displaystyle\mathcal{H}(x,x^{\ast})=\left\langle
x,x^{\ast}\right\rangle ;\medskip\\
\left(  c\right)  & \displaystyle\exists~\left(  u,u^{\ast}\right)  \in
Dom\left(  \partial\mathcal{H}\right)  \text{ such that}\medskip\\
& \displaystyle\quad\quad\quad\quad\quad\left(  x^{\ast},x\right)  \in
\partial\mathcal{H}\left(  u,u^{\ast}\right)  \ \text{and}\ \left\langle
u-x,u^{\ast}-x^{\ast}\right\rangle \geq0;\medskip\\
\left(  d\right)  & \left(  x^{\ast},x\right)  \in\partial\mathcal{H}\left(
x,x^{\ast}\right)  .
\end{array}
\]

\end{theorem}

\begin{proof}
It's not difficult to show that $\left(  b\right)  \Leftrightarrow\left(
a\right)  \Rightarrow\left(  d\right)  \Rightarrow\left(  c\right)
\Rightarrow\left(  a\right)  $. \medskip\hfill
\end{proof}

\begin{corollary}
\label{c2}Let $\mathbb{X}$ be a Banach space and $A:\mathbb{X}%
\rightrightarrows\mathbb{X}^{\ast}$ be a maximal monotone operator. Then%
\[
0\in R\left(  J_{\mathbb{X}}+A\right)  \quad\Longleftrightarrow\quad\left(
0,0\right)  \in R\left(  J_{\mathbb{X}}\otimes J_{\mathbb{X}}^{-1}%
+\partial\mathcal{H}_{A}\right)
\]

\end{corollary}

\begin{proof}
Since $J_{\mathbb{X}}\left(  -x\right)  =-J_{\mathbb{X}}\left(  x\right)  ,$
then we clearly have the following equivalences: $0\in R\left(  J_{\mathbb{X}%
}+A\right)  $ $\Longleftrightarrow$ $\exists~\left(  x,x^{\ast}\right)  \in A$
such that $-x^{\ast}\in J_{\mathbb{X}}\left(  x\right)  $ $\Longleftrightarrow
$ $\exists~\left(  x,x^{\ast}\right)  \in\mathbb{X}\times\mathbb{X}^{\ast}$
such that $\left(  0,0\right)  \in\left(  -x^{\ast},-x\right)  +\partial
\mathcal{H}\left(  x,x^{\ast}\right)  $ and $\left(  -x^{\ast},-x\right)
\in\left(  J_{\mathbb{X}}\left(  x\right)  ,J_{\mathbb{X}}^{-1}\left(
x^{\ast}\right)  \right)  $ $\Longleftrightarrow$ $\left(  0,0\right)  \in
R\left(  J_{\mathbb{X}}\otimes J_{\mathbb{X}}^{-1}+\partial\mathcal{H}%
_{A}\right)  .$ \bigskip\hfill
\end{proof}

Now from Corollary \ref{c1} we have for all $\lambda>0,$%
\[
\overline{R\left(  J_{\mathbb{X}\times\mathbb{X}^{\ast}}+\lambda
\partial\mathcal{H}_{A}\right)  }=\mathbb{X}^{\ast}\times\mathbb{X}^{\ast\ast}%
\]
When $\mathbb{X}$ is a reflexive Banach space $J_{\mathbb{X}\times
\mathbb{X}^{\ast}}\left(  x,x^{\ast}\right)  =J_{\mathbb{X}}\left(  x\right)
\otimes J_{\mathbb{X}}^{-1}\left(  x^{\ast}\right)  $ and therefore, by
Corollary \ref{c1},%
\begin{equation}
R\left(  J_{\mathbb{X}}\otimes J_{\mathbb{X}}^{-1}+\lambda\partial
\mathcal{H}_{A}\right)  =\mathbb{X}^{\ast}\times\mathbb{X} \label{s1}%
\end{equation}
In this sequence of ideas we can rewrite simplifying the approach from
\cite{SZ} of Simons and Z\u{a}linescu for the Rockafellar's surjectivity
result of maximal monotone operators, as follows

\begin{theorem}
[Rockafellar]Let $\mathbb{X}$ be a reflexive Banach space. If $A:\mathbb{X}%
\rightrightarrows\mathbb{X}^{\ast}$ is a maximal monotone operator, then
$R\left(  J_{\mathbb{X}}+\lambda A\right)  =\mathbb{X}^{\ast},$ for all
$\lambda>0.$
\end{theorem}

\begin{proof}
Let $\lambda>0$ and $z^{\ast}\in\mathbb{X}^{\ast}.$ Since $A$ is maximal
monotone if and only $\tilde{A}=\lambda A-z^{\ast}$ is maximal monotone, then
to prove $z^{\ast}\in R\left(  J_{\mathbb{X}}+\lambda A\right)  $ is
equivalent to prove $0\in R\left(  J_{\mathbb{X}}+A\right)  .\ $But by
(\ref{s1}) and Corollary \ref{c2} the relation $0\in R\left(  J_{\mathbb{X}%
}+A\right)  $ is equivalent to the true assertion $\left(  0,0\right)  \in
R\left(  J_{\mathbb{X}}\otimes J_{\mathbb{X}}^{-1}+\partial\mathcal{H}%
_{A}\right)  =\mathbb{X}^{\ast}\times\mathbb{X}$. \bigskip\hfill
\end{proof}

Finally we note (see \cite{gossez}) that in the case of a non-reflexive Banach
space $\mathbb{X}$, there exists a maximal monotone operator $A:\mathbb{X}%
\rightrightarrows\mathbb{X}^{\ast}$ and $\lambda>0$ such that $\overline
{R\left(  J_{\mathbb{X}}+\lambda A\right)  }\subsetneq\mathbb{X}^{\ast}$.

\addcontentsline{toc}{section}{References}

\end{document}